\documentclass[11pt]{article}
\usepackage{amsfonts,amsthm,amsmath}
\usepackage{authblk}
\newcommand{\zed}{\ensuremath{\mathbb Z}}
\newcommand{\eff}{\ensuremath{\mathbb F}}
\newtheorem{Theorem}{Theorem}[section]

\newtheorem{Example}{Example}[section]
\newtheorem{Remark}{Remark}[section]
\newtheorem{Lemma}[Theorem]{Lemma}
\newtheorem{Corollary}[Theorem]{Corollary}

\title {\textbf{All or Nothing at All}\footnote{We have ``borrowed'' the title of this paper from the
classic song of the same name written by Altman and Lawrence in 1939.
It was recorded by Frank Sinatra and the Harry James Orchestra
in 1939, and became a huge hit in 1943.}}
\author[1]{Paolo D'Arco}
\author[2]{Navid Nasr Esfahani}
\author[2]{Douglas~R.~Stinson\thanks{D.~Stinson's research is supported by NSERC discovery grant 203114-11.}}
\affil[1]{Dipartimento di Informatica, Universit\`{a} degli Studi di Salerno, 
84084 Fisciano (SA), Italy}
\affil[2]{David R.\ Cheriton School of Computer Science, University of Waterloo,
Waterloo, Ontario N2L 3G1, Canada}
\date{\today}

\begin{document}
\maketitle

\begin{abstract}
We continue a study of unconditionally secure all-or-nothing transforms (AONT)
begun in \cite{St}. An AONT is a bijective mapping that constructs
$s$ outputs from $s$ inputs. We consider the security of $t$ inputs,
when $s-t$ outputs are known. Previous work concerned the case $t=1$;
here we consider the problem for general $t$, focussing on the case $t=2$.
We investigate constructions of binary matrices for which the desired 
properties hold with the maximum probability. Upper bounds on
these probabilities are obtained
via a quadratic programming approach, while lower bounds can be obtained
from combinatorial constructions based on symmetric BIBDs and cyclotomy.
We also report some results on  exhaustive searches and random constructions
for small values of $s$.
\end{abstract}

\section{Introduction}
\label{intro.sec}

Rivest defined all-or-nothing transforms in \cite{R} in the setting of computational
security. Stinson considered unconditionally secure all-or-nothing transforms in \cite{St}. Here we extend
some of the results in \cite{St} by considering more general types of
unconditionally secure all-or-nothing transforms.

Let $X$ be a finite set, called an {\em alphabet}.
Let $s$ be a positive integer, and suppose that $\phi : X^s \rightarrow X^s$.
We will think of $\phi$ as a function that maps
an input $s$-tuple, say ${\mathbf x} = (x_1, \dots , x_s)$, to an
output $s$-tuple, say ${\mathbf y} = (y_1, \dots , y_s)$, where
$x_i,y_i \in X$ for $1 \leq i \leq s$.
Informally, the function $\phi$ is an unconditionally secure 
{\em all-or-nothing transform} provided that 
the following properties are satisfied:
\begin{enumerate}
\item $\phi$ is a bijection.
\item If any $s-1$ of the $s$ output values
$y_1, \dots , y_s$ are
fixed, then the value of any one input value $x_i$ ($1 \leq i \leq s$) is
completely undetermined, in an information-theoretic sense.
\end{enumerate}
We will denote such a function as
an $(s,v)$-AONT, where $v = |X|$.

The above definition can be rephrased in terms of the entropy function,
$\mathsf{H}$, as follows.  Let 
$X_1, \dots , X_s, Y_1, \dots , Y_s$ be random variables taking on
values in the finite set $X$.  (The variables $X_1, \dots , X_s$ need
not be independent, or uniformly distributed.)
Then these $2s$ random variables define an AONT provided that
the following conditions are satisfied:
\begin{enumerate}
\item $\mathsf{H}( Y_1, \dots , Y_s \mid X_1, \dots , X_s) = 0$.
\item $\mathsf{H}( X_1, \dots , X_s \mid Y_1, \dots , Y_s) = 0$.
\item $\mathsf{H}( X_i  \mid Y_1, \dots , Y_{j-1} , Y_{j+1}, \dots , Y_s) = \mathsf{H}( X_i)$
for all $i$ and $j$ such that $1 \leq i \leq s$ and $1 \leq j \leq s$ .
\end{enumerate}

Rivest \cite{R} defined AONT  to provide a mode of operation for  block ciphers 
that would require the decryption of all blocks
of an encrypted message in order to determine
any specific single block of plaintext. He called it the ``package transform''.
The method is very simple and elegant. Suppose
we are given $s$ blocks of plaintext, $(x_1, \dots , x_s)$. First, we apply an AONT, computing
$(y_1, \dots , y_s) = \phi (x_1, \dots , x_s)$. Then we encrypt 
$(y_1, \dots , y_s)$ using a block cipher. At the receiver's end, the ciphertext is decrypted,
and then the inverse transform $\phi^{-1}$ is applied to restore the $s$ plaintext blocks.
Note that the transform $\phi$ is not secret. Extensions of this technique are studied in 
\cite{CT,Des}.

All-or-nothing transforms have turned out to have numerous applications in 
cryptography and security. Here are some representative examples:
\begin{itemize}
\item exposure-resilient functions \cite{CDHKS}
\item network coding \cite{CCGCB,GLLY}
\item secure data transfer \cite{VAS}
\item anti-jamming techniques \cite{PL}
\item secure distributed cloud storage \cite{LWXH,SPK}
\item query anonymization for location-based services \cite{ZL}.
\end{itemize}

We note that the above definition of an unconditionally secure AONT does not say anything
regarding partial information that might be revealed about more than one of the $s$ input values.
For example, it does not rule out the possibility of determining
the exclusive-or of two input values, given some relatively small number of output values.
This motivates the following more general definition.
Let $1 \leq t \leq s$. We will say that 
$\phi$ is a 
{\em $t$-all-or-nothing transform} provided that 
the following properties are satisfied:
\begin{enumerate}
\item $\phi$ is a bijection.
\item If any $s-t$ of the $s$ output values
$y_1, \dots , y_s$ are
fixed, then  any $t$ of the input values $x_i$ ($1 \leq i \leq s$) are
completely undetermined, in an information-theoretic sense.
\end{enumerate}
We will denote such a function as
a $(t,s,v)$-AONT, where $v = |X|$.
Note that the original definition corresponds to a $1$-all-or-nothing transform.

This definition can also be rephrased in terms of the entropy function.  
As before, let 
$X_1, \dots , X_s, Y_1, \dots , Y_s$ be random variables taking on
values in the finite set $X$.  
These $2s$ random variables 
define a $t$-AONT provided that
the following conditions are satisfied:
\begin{enumerate}
\item $\mathsf{H}( Y_1, \dots , Y_s \mid X_1, \dots , X_s) = 0$.
\item $\mathsf{H}( X_1, \dots , X_s \mid Y_1, \dots , Y_s) = 0$.
\item For all $\mathcal{X} \subseteq \{X_1, \dots , X_s\}$ with 
$|\mathcal{X}|  = t$, and for all $\mathcal{Y} \subseteq \{Y_1, \dots , Y_s\}$ with 
$|\mathcal{Y}|  = t$, it holds that 
 \begin{equation}
 \label{t-AONT.eq}
 \mathsf{H}( \mathcal{X}  \mid \{ Y_1, \dots , Y_s\}  \setminus \mathcal{Y} ) = \mathsf{H}(\mathcal{X}).
 \end{equation}
\end{enumerate}

\subsection{Organization of the Paper}

The rest of this paper is organized as follows. In Section
\ref{linear.sec}, we give our basic result that characterizes
linear AONT in terms of matrices having invertible submatrices.
We also give a construction using Cauchy matrices  over a finite field
$\eff_q$, which is applicable provided
that $q \geq 2s$. It turns out that it is impossible to 
construct linear
AONT over $\eff_2$, so an interesting question is how ``close''
one can get to an AONT in this setting.
In Section \ref{2.sec}, we give some preliminary results and analyze one
infinite class of matrices. In Section \ref{nonexist}, we derive an upper bound 
on the maximum number of invertible 2 by 2 submatrices of an invertible $s$ by $s$ $0-1$ matrix
(this is relevant for the study of $2$-AONT). We use a method based on 
quadratic programming to prove our bound. In Section \ref{const.sec},
we discuss five construction methods for invertible $s$ by $s$ $0-1$ matrices
containing a large number of invertible 2 by 2 submatrices. The five methods are
\begin{enumerate}
\item exhaustive search,
\item a random construction,
\item a recursive construction,
\item a construction using symmetric balanced incomplete block designs (SBIBDs), and
\item a construction based on cyclotomy.
\end{enumerate}
We achieve our best asymptotic results from SBIBDs, where we have an infinite class of 
examples that are close to the upper bound derived in the previous section.
In Section \ref{general.sec}, we turn to arbitrary (i.e., linear or nonlinear)
AONT, and describe some connections with orthogonal arrays.
Finally, Section \ref{summary.sec} is a brief summary.


\section{Linear AONT}
\label{linear.sec}

We are mainly going to consider linear transforms.
Let $\eff_q$ be a finite field of order $q$.
An AONT with alphabet $\eff_q$ is {\em linear}
if each $y_i$ is an $\eff_q$-linear function of 
$x_1, \dots , x_s$. Then, we can write 
${\mathbf y} = \phi({\mathbf x}) = {\mathbf x}M^{-1}$
and  ${\mathbf x} = \phi^{-1}({\mathbf y}) = {\mathbf y}M$,
where $M$ is an invertible $s$ by $s$ matrix with entries from $\eff_q$. 

We will also be interested in functions that satisfy the condition (\ref{t-AONT.eq}) for
certain (but not necessarily all) pairs $\mathcal{X}, \mathcal{Y}$. This will be particularly
relevant in the  case where $\phi$ is a binary linear transformation. 
More specifically, suppose $q = 2^r$ for some $r\geq 1$ and $M$ is defined over the subfield
$\eff_2$ (so $M$ is a $0-1$ matrix). This could be desirable from an efficiency
point of view, because the only operations required to compute 
the transform are exclusive-ors of bitstrings. 
However, it turns out that there are no nontrivial 1-AONT
(a fact that was already observed in \cite{St}). So it is a reasonable and interesting
problem to study how close we can get to an AONT in this setting.
We will give a precise answer to this question for $t=1$ in Theorem \ref{R1.thm};
much of the rest of this paper will study the corresponding problem when $t=2$.

For $I,J \subseteq \{1, \dots , s\}$, define $M(I,J)$ to be the $|I|$ by $|J|$ 
submatrix of $M$ induced by the {\it columns} in $I$ and the {\it rows} in $J$.
The following lemma characterizes 
linear all-or-nothing transforms in terms of properties of the matrix $M$.
This lemma can be considered to be a generalization of \cite[Theorem 2.1]{St}.

\begin{Lemma}
\label{linear}
Suppose that $q$ is a prime power and $M$ is 
an invertible $s$ by $s$ matrix with entries from $\eff_q$. 
Let $\mathcal{X} \subseteq \{X_1, \dots , X_s\}$,  
$|\mathcal{X}|  = t$, and let $\mathcal{Y} \subseteq \{Y_1, \dots , Y_s\}$,  
$|\mathcal{Y}|  = t$.
Then the function $\phi({\mathbf x}) = {\mathbf x}M^{-1}$
satisfies (\ref{t-AONT.eq}) with respect to $\mathcal{X}$ and $\mathcal{Y}$ if and only if 
the submatrix $M(I,J)$ is invertible, where
$I = \{i: X_i \in \mathcal{X}\}$ and $J = \{j: Y_j \in \mathcal{Y}\}$.
\end{Lemma}

\begin{proof}
Let ${\mathbf x'} = (x_i : i \in I)$. We have ${\mathbf x'} = {\mathbf y} M(I,\{1,\dots , s\})$.
Now assume that $y_j$ is fixed for all $j \in J$. Then we can write
${\mathbf x'} = {\mathbf y'} M(I,J) + {\mathbf c}$, where ${\mathbf y'} = (y_j : j \in J)$
and ${\mathbf c}$ is a vector of constants.
If $M(I,J)$ is invertible, then ${\mathbf x'}$ is completely undetermined, in the sense that
${\mathbf x'}$ takes on all values in $(\eff_q)^t$ as ${\mathbf y'}$ 
varies over $(\eff_q)^t$. On the other hand, if $M(I,J)$ is not invertible, then
${\mathbf x'}$ can take on only $(\eff_q)^{t'}$ possible values, where $\mathsf{rank}(M(I,J)) = t' < t$.
\end{proof}

An $s$ by $s$ \emph{Cauchy matrix} can be defined over $\eff_q$ 
if $q\geq 2s$. Let $a_1, \dots , a_s,b_1, \dots , b_s$ be distinct elements of
$\eff_q$. Let $c_{ij} = 1/(a_i-b_j)$, for $1 \leq i \leq s$ and $1 \leq j \leq s$.
Then $C = (c_{ij})$ is the Cauchy matrix defined by the sequence  $a_1, \dots , a_s,b_1, \dots , b_s$.
The most important property of a Cauchy matrix $C$ is that any square submatrix of $C$ (including $C$
itself) is invertible over $\eff_q$.

Cauchy matrices were briefly mentioned in \cite{St} as a possible method of 
constructing AONT. However, they are particularly relevant in light of the stronger
definitions we are now investigating. To be specific, 
Cauchy matrices immediately yield the strongest possible all-or-nothing transforms, 
as stated in the following theorem.

\begin{Theorem}
Suppose $q$ is a prime power and $q \geq 2s$. Then there is a linear transform
that is simultaneously a $(t,s,q)$-AONT for all $t$ such that $1 \leq t \leq s$.
\end{Theorem}

\section{Linear Transforms over $\eff _2$}
\label{2.sec}


\begin{Remark}
In the remainder of the paper, when we discuss invertibility of a matrix, we mean
invertibility over $\eff_2$.
\end{Remark}

There is no Cauchy matrix over $\eff_2$ if $s > 1$. In fact, it is easy to see that
there is no linear $(1,s,2)$-AONT if $s > 1$. This is because every entry of $M$ must equal $1$
(in order that the $1$ by $1$ submatrices of $M$ are invertible). But then $M$ itself is not
invertible. This motivates trying to determine how close we can get to a $(1,s,2)$-AONT,
or more generally, to a $(t,s,2)$-AONT, for a given $t$, $1 \leq t \leq s$. 

For future reference, we record the  $2$ by $2$ invertible $0-1$ matrices.

\begin{Lemma}
\label{2x2.lem}
A $2$ by $2$ $0-1$ matrix is  invertible if and only if it is one of the following
six matrices:
\[ 
\begin{array}{cccccc}
\left(
\begin{array}{cc}
1 & 1 \\
1 & 0 
\end{array}
\right) &
\left(
\begin{array}{cc}
1 & 1 \\
0 & 1 
\end{array}
\right) &
\left(
\begin{array}{cc}
0 & 1 \\
1 & 1 
\end{array}
\right) &
\left(
\begin{array}{cc}
1 & 0 \\
1 & 1 
\end{array}
\right) &
\left(
\begin{array}{cc}
1 & 0 \\
0 & 1
\end{array}
\right) &
\left(
\begin{array}{cc}
0 & 1 \\
1 & 0 
\end{array}
\right).
\end{array}
\]
\end{Lemma}

We first consider an  example.

\begin{Example}
\label{exam-3}
Define a $3$ by $3$ matrix:
\[M = 
\left(
\begin{array}{ccc}
1 & 1 & 1\\
1 & 0 & 1\\
1 & 1 & 0
\end{array}
\right) .
\]
It is clear that seven of the nine $1$ by $1$ submatrices of $M$ are invertible
(namely, the ``1'' entries).
There are nine $2$ by $2$ submatrices of $M$ and seven of them
are seen to be invertible, from Lemma \ref{2x2.lem}. The only non-invertible $2$ by $2$ submatrices 
are $M(\{1,3\},\{1,2\})$ and $M(\{1,2\},\{1,3\})$.
Finally, $M$ itself is invertible.
\end{Example}

It seems natural to quantify the ``closeness'' of $M$ to an all-or-nothing transform
by considering the ratio of invertible square submatrices 
to the total number of square submatrices (of a given size $t$).
Therefore, for an $s$ by $s$ invertible $0-1$ matrix $M$ and for  
$1 \leq t \leq s$, we  define
\[
N_{t}(M) = \text{number of invertible $t$ by $t$ submatrices of $M$}
\]
and
\[
R_{t}(M) = \frac{N_t(M)}{\binom{s}{t}^2}.
\]
We refer to $R_{t}(M)$ as the \emph{$t$-density} of the matrix $M$.
For $1 \leq t \leq s$, we also define 
\[ R_t(s) = \max \{  R_{t}(M) : \text{$M$ is an $s$ by $s$ invertible $0-1$ matrix}\} .\]
$R_t(s)$ denotes the maximum $t$-density of any $s$ by $s$ invertible $0-1$ matrix.

For the matrix $M$ from Example \ref{exam-3}, we have $R_{1}(M) = 7/9$
and $R_{2}(M) = 7/9$.

\begin{Lemma}
\label{invertible1}
Suppose $M$ is an $s$ by $s$ $0-1$ matrix having at most $s-2$ zero entries. 
Then $M$ is not invertible.
\end{Lemma}
\begin{proof}
There must exist at least two columns of $M$ that do not contain a zero entry.
These two columns are identical, so they are linearly dependent.
\end{proof}

\begin{Lemma}
\label{invertible2}
Suppose $M = J_s - I_s$, where $I_s$ denotes the $s$ by $s$  identity matrix 
and $J_s$ denotes the $s$ by $s$ matrix in which every entry is equal to one.
Then $M$ is invertible over $\eff_2$ if and only if $s$ is even.
\end{Lemma}
\begin{proof}
If $s$ is even, then it is easy to check that 
$M^{-1} = M$. If $s$ is odd, then observe that the sum of all the columns of
$M$ yields the zero-vector, so we have a dependence relation among the columns
of $M$.
\end{proof}

\begin{Lemma}
\label{invertible3}
Suppose $M$ is an $s$ by $s$ $0-1$ matrix having exactly $s-1$ zero entries. 
Then $M$ is invertible over $\eff_2$ if and only if the zero entries 
occur in $s-1$ different rows and in $s-1$ different columns.
\end{Lemma}
\begin{proof}
First suppose that there are at least two zero entries in a specific column of $M$.
Then there must exist at least two columns of $M$ that do not contain a zero entry,
and $M$ is not invertible, as in Lemma \ref{invertible1}. A similar conclusion
holds if there exist at least two zero entries in a specific row of $M$.
Therefore we can restrict our attention to the case where the zero entries 
occur in $s-1$ different rows and in $s-1$ different columns. We will show that $M$ is invertible in
this case. 

By permuting rows and columns if necessary (which does not affect invertibility), 
we can assume that $M = (m_{ij})$ has the form
\begin{equation}
\label{M.eq}
M = 
\left(
\begin{array}{cccccc}
1 & 1 & 1 & 1 & \hdots  & 1\\
1 & 0 & 1 & 1 & \hdots  & 1\\
1 & 1 & 0 & 1 & \hdots  & 1\\
\vdots & \vdots & \vdots & \vdots & \ddots & \vdots \\
1 & 1 & 1 & 1 & \hdots & 0
\end{array}
\right) .
\end{equation}

We will prove that $M$ is invertible by induction on $s$. Clearly we can use $s=1$ as a base case.
Now we assume $s \geq 2$ and we evaluate $\det M$ over $\eff_2$ by using a cofactor expansion along the first column.
This yields
\begin{eqnarray*}
 \det M &=& \sum_{i=1}^{s} m_{i1} \times \det(M_{i1}) \\
 &=& \sum_{i=1}^{s} \det(M_{i1}) ,
 \end{eqnarray*} 
 where $M_{i1}$ is the {\it minor} formed by deleting row $i$ and column $1$ of $M$.

We consider two cases, depending on whether $s$ is even or odd. First, suppose that $s$ is odd.
Here, $\det(M_{11}) = 1$ from Lemma \ref{invertible2}, and $\det(M_{i1}) = 1$ for $2 \leq i \leq s$ by induction.
It follows that $\det(M) = s \bmod 2 = 1$.

Now let $s$ be even. We have that $\det(M_{11}) = 0$ from Lemma \ref{invertible2}, and $\det(M_{i1}) = 1$ for $2 \leq i \leq s$ by induction.
It follows that $\det(M) = (s-1) \bmod 2 = 1$. By induction, the proof is complete.
\end{proof}

The following result is an immediate corollary of Lemmas \ref{invertible1} and \ref{invertible3}.
\begin{Theorem}
\label{R1.thm}
For all $s \geq 1$, we have $R_1(s) = 1 - \frac{s-1}{s^2}$.
\end{Theorem}

\begin{Remark}
It was shown in 
\cite{St} that $R_1(s) \geq 1 - \frac{1}{s}$ when $s$ is even. This was 
based on using the matrix $J_s - I_s$ as a transform. 
Theorem \ref{R1.thm} is a slight improvement, and it holds for all values of $s$.
\end{Remark}

\begin{Example}
\label{exam-4}
Consider the  $4$ by $4$ matrix given by (\ref{M.eq}):
\[M = 
\left(
\begin{array}{cccc}
1 & 1 & 1 & 1\\
1 & 0 & 1 & 1\\
1 & 1 & 0 & 1\\
1 & 1 & 1 & 0
\end{array}
\right) .
\]
Here, we can verify using Lemmas
\ref{invertible1}, \ref{invertible2} and \ref{invertible3} that $R_{1}(M) = 13/16$,
$R_{2}(M) = 24/36 = 2/3$ and $R_{3}(M) = 9/16$.
\end{Example}

In fact, it is possible to compute all the values
$R_t(M)$ for the $s$ by $s$ matrix $M$ given in (\ref{M.eq}).
There are $\binom{s}{t}^2$ submatrices $N$ of $M$ of dimensions $t$ by $t$. 
From the structure of $M$, and from Lemmas 
\ref{invertible1}, \ref{invertible2} and \ref{invertible3}, we see that 
a $t$ by $t$ submatrix $N$ is invertible if and only if one of the following 
conditions holds:
\begin{enumerate}
\item $N$ contains $t-1$ zero entries, or
\item $t$ is even and $N$ contains $t$ zero entries.
\end{enumerate}
If we can count the number of submatrices of this form, then we can compute
$R_t(M)$. But this is not hard to do.

\begin{Lemma}
The $s$ by $s$ matrix $M$ given in (\ref{M.eq}) has exactly $\binom{s-1}{t-1}(1 + (s-t+1)(s-t))$ 
submatrices that
contain exactly $t-1$ zero entries.
\end{Lemma}

\begin{Lemma}
The $s$ by $s$ matrix $M$ given in (\ref{M.eq}) has exactly $\binom{s-1}{t}$ submatrices that
contain exactly $t$ zero entries.
\end{Lemma}

So we now obtain the following.

\begin{Theorem}
\label{M.bound}
Let $M$ be the $s$ by $s$ matrix given in (\ref{M.eq}) and let $1 \leq t \leq s-1$.
If $t$ is odd, then
\[N_t(M) = \binom{s-1}{t-1}(1 + (s-t+1)(s-t)).\]
If $t$ is even, then 
\[N_t(M) = \binom{s-1}{t} + \binom{s-1}{t-1}(1 + (s-t+1)(s-t)).\]
\end{Theorem}

Theorem \ref{M.bound} also provides (constructive) lower bounds on $R_t(s)$ for all values of 
$t \leq s$. We do not claim that these bounds are necessarily good asymptotic bounds, however.
Even for $t=2$, we get $R_2(M) \rightarrow 0$ as $s \rightarrow \infty$,
since $\binom{s-1}{t-1}(1 + (s-t+1)(s-t)) \in \Theta(s^3)$ and $\binom{s}{t}^2 \in \Theta(s^4)$.
This suggests looking for constructions which will yield constant lower bounds on $R_2(s)$.
On the other side, we would also like to find good upper bounds on $R_2(s)$. 

\section{Upper Bounds for $R_2(s)$}
\label{nonexist}

We first establish an easy upper bound for $R_2(s)$.
This bound follows from the following lemma.

\begin{Lemma}
\label{2-lemma}
Any $2$ by $s$ $0-1$ matrix contains at most $s^2/3$ invertible $2$ by $2$ submatrices.
\end{Lemma}

\begin{proof}
Let $N$ be any $2$ by $s$ $0-1$ matrix.
Consider the $2$ by $1$  submatrices of $N$. 
Suppose there are $a_0$ occurrences of 
$\left( \begin{array}{c} 0 \\ 0 \end{array} \right)$,
$a_1$ occurrences of 
$\left( \begin{array}{c} 0 \\ 1 \end{array} \right)$,
$a_2$ occurrences of 
$\left( \begin{array}{c} 1 \\ 0 \end{array} \right)$,
and $a_3$ occurrences of 
$\left( \begin{array}{c} 1 \\ 1 \end{array} \right)$.
Of course $a_0 + a_1 + a_2 + a_3  = s$.
From Lemma \ref{2x2.lem}, the number of invertible $2$ by $2$ submatrices in $N$ is easily seen to be
$a_1a_2 + a_1a_3 + a_2a_3$. This expression is maximized when $a_0 = 0$, $a_1 = a_2 = a_3 =s/3$, yielding
$3(s/3)^2 = s^2 / 3$ invertible $2$ by $2$ submatrices.
\end{proof}

\begin{Theorem}
\label{upperbound}
For any $s \geq 2$, it holds that
\[ R_2(s) \leq \frac{2s}{3(s-1)}.\]
\end{Theorem}

\begin{proof}
From Lemma \ref{2-lemma}, in any two rows of $M$ there are at most $s^2/3$
invertible $2$ by $2$ submatrices.
Now, in the entire matrix $M$, there are $\binom{s}{2}$ ways to choose two rows, and
there are $\binom{s}{2}^2$ submatrices of order $2$. This immediately yields
\[ R_2(s) \leq  \frac{\binom{s}{2}(s^2 / 3)}{\binom{s}{2}^2} = \frac{2s}{3(s-1)}.\]
\end{proof}

\begin{Example}
\label{exam3}
When $s=3$, we only get the trivial upper bound $R_2(3) \leq 1$ from Theorem \ref{upperbound}.
Consider the matrix
\[M = 
\left(
\begin{array}{ccc}
0 & 1 & 1\\
1 & 0 & 1\\
1 & 1 & 0
\end{array}
\right) .
\]
It is clear from the proof of
Theorem \ref{upperbound} that all nine $2$ by $2$ submatrices of $M$ are invertible, 
and $M$ is the only $3$ by $3$ matrix with this property.
However, $M$ is not itself invertible, so we can conclude that 
$R_2(3) \leq 8/9$. Example \ref{exam-3} shows that $R_2(3) \geq 7/9$.

In fact, we can show that $R_2(3) = 7/9$. Suppose that $R_2(3) = 8/9$.
Let $R_2(M) = 8/9$. Then we can assume that the first two rows of $M$
contain three invertible $2$ by $2$ submatrices, the first and third rows of $M$
contain three invertible $2$ by $2$ submatrices, and the last two rows of $M$
contain two invertible $2$ by $2$ submatrices.
By permuting columns, the first two rows of $M$ look like:
\[ \left(
\begin{array}{ccc}
0 & 1 & 1\\
1 & 0 & 1
\end{array}
\right).
\] 
In order that the first and third rows 
contain three invertible $2$ by $2$ submatrices, the third row must be 
$1 \: 0 \: 1$ or $1 \: 1 \: 0$. In the first case,  the last two rows of $M$
contain no invertible $2$ by $2$ submatrices, and in the second case, the  last two rows of $M$
contain three invertible $2$ by $2$ submatrices.
We conclude that $R_2(3) < 8/9$, so $R_2(3) = 7/9$.
\end{Example}

\begin{Example}
\label{exam4}
When $s=4$, we get $R_2(4) \leq 8/9$ from Theorem \ref{upperbound}.
Consider the matrix $M = J_4 - I_4$.
$M$ is  invertible from Lemma \ref{invertible2}.
It is easy to check that $30$ of the $2$ by $2$ submatrices of $M$ are invertible.  
Therefore, $R_2(4) \geq 5/6$. 

We can in fact show that $R_2(4) = 5/6$, as follows.
Suppose $R_2(4) > 5/6$. Then there is a $4$ by $4$ $0-1$ matrix $M$ having
at least $31$ invertible $2$ by $2$ submatrices. There are six pairs of rows in $M$,
and $31 > 6 \times 5$, 
so there is at least one pair of rows that contains six invertible $2$ by $2$ submatrices.
But this contradicts Lemma \ref{2-lemma}, where it is shown that the maximum number of
$2$ by $2$ submatrices in two given rows is at most $4^2 / 3 = 16/3 < 6$.
\end{Example}

We next present a generalization of Theorem \ref{upperbound} that leads to an improved upper bound on
$R_2(s)$. The proof of Theorem \ref{upperbound} was
based on upper-bounding the number of invertible $2$ by $2$ submatrices in any two rows
of an $s$ by $s$ matrix $M$. Here we instead determine an upper bound on the number of 
invertible $2$ by $2$ submatrices in any four rows
of $M$. (It turns out that considering three rows at a time yields the same bound
as Theorem  \ref{upperbound}, so we skip directly to an analysis of four rows at a time.)

Label the  non-zero vectors in $\{0,1\}^4$ in lexicographic order as follows:
$b_0 = (0,0,0,0)$, $b_1 = (0,0,0,1)$, $b_2 = (0,0,1,0)$, $b_3 = (0,0,1,1)$, $\dots$,
$b_{15} = (1,1,1,1)$. For $1 \leq i , j \leq 15$, define
$c_{ij}$ to be the number of invertible $2$ by $2$ submatrices in the 
$4$ by $2$ matrix 
$\left( \begin{array}{c|c}
b_i^T & b^T_j
\end{array}
\right)$.
Let $C = (c_{ij})$; note that $C$ is a $15$ by $15$ symmetric matrix with 
zero diagonal such that every off-diagonal
element is a positive integer. This matrix $C$ is straightforward to compute and
it is presented in Figure \ref{C.fig}.

Now define ${\bf z} = (z_1, \dots , z_{15})$ and consider the following quadratic program $\mathcal{Q}$:
\begin{center}
\begin{tabular}{|ll|}
\hline
Maximize & $\frac{1}{2} {\bf z} C {\bf z}^T$ \rule{0pt}{2.6ex} \\ 
subject to & $\sum _{i=1}^{15} z_i \leq 1$
and  $z_i \geq 0$, for all $i$, $1 \leq i \leq 15$.\rule[-1.2ex]{0pt}{0pt}\\
\hline
\end{tabular}
\end{center}
We have the following result.

\begin{Theorem}
\label{better.thm}
For any integer $s \geq 4$, it holds that
\[ R_2(s) \leq \frac{\gamma s}{3(s-1)},\]
where $\gamma$ denotes the optimal solution to $\mathcal{Q}$.
\end{Theorem}

\begin{proof}
Let $M$ be any $s$ by $s$ $0-1$ matrix.
Consider any four rows of $M$, say the first four rows without loss of generality,
and denote the resulting $4$ by $s$ submatrix by $M'$.
For $0 \leq i \leq 15$, suppose there are $a_i$ columns of $M'$ 
that are equal to $b_i^T$. The number $N$ of $2$ by $2$ invertible submatrices of $M'$ 
is equal to $\frac{1}{2} {\bf a} C {\bf a}^T$, where
${\bf a} = (a_1, \dots , a_{15})$ (we can ignore $a_0$ because a zero column does not give rise
to any invertible submatrices). If we now define $z_i = a_i / s$ for all $i$, then
we obtain \[ N = \frac{1}{2} {\bf a} C {\bf a}^T = \frac{s^2}{2} {\bf z} C {\bf z}^T \leq \gamma s^2.\] 

There are $\binom{s}{4}$ ways to choose four rows from $M$. 
The total number of occurrences of invertible $2$ by $2$ submatrices obtained is at most
$\binom{s}{4} \gamma s$. However, each invertible $2$ by $2$ submatrix is included
in exactly $\binom{s-2}{2}$ sets of four rows, so the total number of
invertible $2$ by $2$ submatrices is at most 
\[ \frac{\binom{s}{4} \gamma s^2}{\binom{s-2}{2}}.\]
The total number of $2$ by $2$ submatrices is $\binom{s}{2}^2$, so we obtain the upper bound
\begin{equation}
\label{R_2.eq}
 R_2(s) \leq \frac{\binom{s}{4} \gamma s^2}{\binom{s-2}{2} \binom{s}{2}^2}
= \frac{\gamma s}{3(s-1)}.
\end{equation} 
\end{proof}

\begin{figure}[tb]
\[
C = 
\left(
\begin{array}{ccccccccccccccc}
0 & 1 & 1 & 1 & 1 & 2 & 2 & 1 & 1 & 2 & 2 & 2 & 2 & 3 & 3\\
1 & 0 & 1 & 1 & 2 & 1 & 2 & 1 & 2 & 1 & 2 & 2 & 3 & 2 & 3\\
1 & 1 & 0 & 2 & 3 & 3 & 2 & 2 & 3 & 3 & 2 & 4 & 5 & 5 & 4\\
1 & 1 & 2 & 0 & 1 & 1 & 2 & 1 & 2 & 2 & 3 & 1 & 2 & 2 & 3\\
1 & 2 & 3 & 1 & 0 & 3 & 2 & 2 & 3 & 4 & 5 & 3 & 2 & 5 & 4\\
2 & 1 & 3 & 1 & 3 & 0 & 2 & 2 & 4 & 3 & 5 & 3 & 5 & 2 & 4\\   
2 & 2 & 2 & 2 & 2 & 2 & 0 & 3 & 5 & 5 & 5 & 5 & 5 & 5 & 3\\
1 & 1 & 2 & 1 & 2 & 2 & 3 & 0 & 1 & 1 & 2 & 1 & 2 & 2 & 3\\
1 & 2 & 3 & 2 & 3 & 4 & 5 & 1 & 0 & 3 & 2 & 3 & 2 & 5 & 4\\
2 & 1 & 3 & 2 & 4 & 3 & 5 & 1 & 3 & 0 & 2 & 3 & 5 & 2 & 4\\
2 & 2 & 2 & 3 & 5 & 5 & 5 & 2 & 2 & 2 & 0 & 5 & 5 & 5 & 3\\
2 & 2 & 4 & 1 & 3 & 3 & 5 & 1 & 3 & 3 & 5 & 0 & 2 & 2 & 4\\
2 & 3 & 5 & 2 & 2 & 5 & 5 & 2 & 2 & 5 & 5 & 2 & 0 & 5 & 3\\
3 & 2 & 5 & 2 & 5 & 2 & 5 & 2 & 5 & 2 & 5 & 2 & 5 & 0 & 3\\
3 & 3 & 4 & 3 & 4 & 4 & 3 & 3 & 4 & 4 & 3 & 4 & 3 & 3 & 0
\end{array}
\right).
\]
\caption{The objective function $C$ for the quadratic program}
\label{C.fig}
\end{figure}

In general, it can be difficult to find (global) optimal solutions for
quadratic programs. We were able to  solve our quadratic program $\mathcal{Q}$
using the BARON software \cite{TS} on the NEOS server ({\tt http://www.neos-server.org/neos/}).
The result is that $\gamma = 15/8$ and an optimal solution is given by
$z_7 = z_{11} = z_{13} = z_{14} = 1/4$, $z_i = 0$ if $i \not\in \{7,11,13,14\}$.
It is interesting to observe that this solution corresponds to
the given set of four rows containing only columns consisting of
three 1's and one 0. In fact, when $s=4$, this provides an alternative proof of  
Example \ref{exam4}.

Applying Theorem \ref{better.thm}, we immediately obtain the following
improved upper bound.

\begin{Corollary}
\label{Cor-15}
For any $s \geq 4$, it holds that
\[ R_2(s) \leq \frac{5s}{8(s-1)}.\]
\end{Corollary}
This upper bound is asymptotically equal to $5/8$, which is a definite improvement
over the asymptotic upper bound of $2/3$ obtained from Theorem \ref{upperbound}.

It is of course possible to generalize this approach, by considering $\rho$ rows at a time.
The coefficient matrix $C$ will have $2^{\rho} - 1$ rows and columns. If $\gamma_{\rho}$ denotes the
solution to the related quadratic program, then we obtain the following 
generalization of Theorem \ref{better.thm}.
\begin{Theorem}
\label{better-g.thm}
For any integers $s \geq {\rho} \geq 2$, it holds that
\begin{equation}
\label{general.eq} R_2(s) \leq \frac{4\gamma_{\rho}}{{\rho}({\rho}-1)} \times \frac{s}{s-1}.
\end{equation}
\end{Theorem}

\begin{proof}
The equation (\ref{R_2.eq}) becomes the following:
\[R_2(s) \leq \frac{\binom{s}{{\rho}} \gamma_{\rho} s^2}{\binom{s-2}{{\rho}-2} \binom{s}{2}^2} = 
\frac{4\gamma_{\rho}}{{\rho}({\rho}-1)} \times \frac{s}{s-1}.\]
\end{proof}
The difficulty in obtaining improved bounds using this approach is that the
optimal solutions $\gamma_{\rho}$ of the quadratic programs are hard to compute.

\section{Constructions}
\label{const.sec}

In the next subsections, we consider five possible construction methods for  AONT with good $2$-density.
The first is exhaustive search.
The second is based on choosing each entry independently at random with an appropriate probability.
The third technique is a recursive technique.
The fourth method is based on using incidence matrices of symmetric BIBDs. 
Our fifth and last approach makes use of classical results concerning cyclotomy and cyclotomic numbers.

\subsection{Exhaustive Searches}
\label{exhaustive.sec}

We used an exhaustive search in order to find an invertible $s\times s$ matrix with the maximum possible number of invertible $2\times 2$  submatrices, for $4\le s\le 8$. The algorithm consists of $s$ nested  loops, each iterating over the possible values in a given row of the matrix. There are $2^s$ possibilities for any given row.
However, any permutation of rows and columns does not affect  either the nonsingularity of the matrix or the number of invertible $2\times 2$ submatrices. Therefore, the search algorithm only generated matrices in which each row has at least as many 1's as  the row immediately above it. Also, if two rows have the same number of 1's, the row
having the smaller representation as a binary number would appear higher. 
These two rules enabled us to search only a ${1}/{s!}$ fraction of the search domain. Finally, we partially restricted column permutations by fixing all the 1's in the first row to occur in the rightmost positions. This also
sped up the search process.
	
The computations for $4\le s \le 8$ were executed on one node on the  Cheriton School of Computer Science server, {\tt linux.cs.uwaterloo.ca}, which has a 64 bit AMD CPU, having a 2.6 GHz clock rate. For $s=9$,
we  attempted to use the same algorithm distributed over 256 processors on {\tt grex.westgrid.ca}. But the search was not finished by the end of the 96 hour time limit. However, it did find a solution  with 783 invertible $2\times 2$ submatrices, which is presented in Example \ref{example9}.

\subsection{Random Constructions}
\label{random.sec}

We investigate the expected number of invertible $2$ by $2$ submatrices in a random
$s$ by $s$ $0-1$ matrix $M$. Suppose every entry of $M$ is chosen to be a 1 with probability $\epsilon$,
independent of the values of all other entries. Using Lemma \ref{2x2.lem}, it is 
easy to see that a specified $2$ by $2$ submatrix is invertible
with probability 
\[ 4 \epsilon^3 (1- \epsilon) + 2 \epsilon^2 (1- \epsilon)^2 = 2 \epsilon^2 (1- \epsilon) (2 \epsilon + 1 - \epsilon) = 
2 \epsilon^2 (1- \epsilon^2).\]
This function is maximized by choosing $\epsilon = \sqrt{1/2}$. The expected number of 
invertible $2$ by $2$ submatrices in $M$ is $\frac{1}{2} \binom{s}{2}^2$ 
(leading to an expected $2$-density of $.5$). Unfortunately, this does
not immediately yield an AONT because it seems difficult to ensure that the constructed
matrix is itself invertible. However, this random construction proves to be a useful method to
obtain good small examples. 

\subsection{Recursive Constructions}
 We now investigate the possibility of constructing ``good'' AONT recursively.
 Specifically, we analyze a type of doubling construction in a particular case. 
 We begin with the $(2,4,2)$-AONT from Example \ref{exam4}. Recall that
 this AONT arises from
 the matrix $J_4 - I_4$ and it achieves the optimal result $R_2(4) = 5/6$. We might try
 to use this matrix to construct a $(2,8,2)$-AONT.
 There are various ways in which we could try to do this; we present one 
 method which leads to a reasonably good outcome. Consider the matrix
 \[
 M = \left(
 \begin{array}{c|c}
 J_4 - I_4 & J_4 - I_4\\ \hline
 J_4 - I_4 & J_4
 \end{array}
 \right).
 \]
 We first need to show that $M$ is invertible. We show that $\det (M) = 1$ as
 follows. Consider a matrix of the form \[
 M = \left(
 \begin{array}{c|c}
 A & B\\ \hline
 C & D
 \end{array}
 \right),
 \]
 where $A,B,C,D$ are square submatrices and $CD = DC$. In this case, it is known
 that $\det (M) = \det(AD - BC)$. 
 
 In our construction, we have 
 $CD = DC = 3J_4 = J_4$, so this formula can be applied.
 We have $A = B = C = J_4-I_4$ and it is easy to check that 
 $BC = I_4$, $AD = J_4$. Therefore $AD - BC = J_4 - I_4$ and 
 $\det (M) = \det(AD - BC) =  \det(J_4 - I_4) = 1$.
 
 Next, we proceed to compute the number of $2$ by $2$ invertible submatrices of $M$.
 We do this by looking at pairs of rows of $M$, say row $i$ and row $j$, and computing the relevant numbers
 $a_0, a_1,a_2 , a_3$ in each case (where we are using the notation from the proof of
 Lemma \ref{2-lemma}).
 We tabulate the results in Table \ref{tab1}.
 
 \begin{table}[tb]
 \caption{$2$ by $2$ invertible submatrices of $M$}
 \label{tab1}
 \begin{center}
 \begin{tabular}{c|c|c}
 $i,j$ & $a_0, a_1,a_2 , a_3$  & \# invertible submatrices \\ \hline
 $ 1 \leq i < j \leq 4$ & $a_1 = 2, a_2 = 2, a_3 = 4$ & 20\\
 $ 5 \leq i < j \leq 8$ & $a_1 = 1, a_2 = 1, a_3 = 6$ & 13\\
 $ 1 \leq i \leq 4, 5 \leq j \leq 8, j \neq i + 4$ & $a_1 = 2, a_2 = 1, a_3 = 5$ & 17\\
 $ 1 \leq i \leq 4, j = i + 4$ & $a_0 = 1, a_1 = 1, a_3 = 6$ & 6
 \end{tabular}
 \end{center}
 \end{table}
 
 The number of occurrences of the four cases enumerated in Table \ref{tab1} is
 (respectively) $6,6,12$ and $4$. Therefore, 
\[ N_2(M) = 6 \times 20 + 6 \times 13 + 12 \times 17 + 4 \times 6 = 426.\]
Finally, we compute 
\[ R_2(M) = \frac{426}{\binom{8}{2}^2} = \frac{426}{784} = .5434.\]
Summarizing, we have the following.
\begin{Theorem}
$N_2(8) \geq 426$ and $R_2(8) \geq .5434$.
\end{Theorem}

It is interesting to note that this recursive construction yields
a better result than the direct constructions considered previously.
For example, if $M = J_8 - I_8$, then we only get that $N_2 \geq 364$.
Also, Theorem \ref{M.bound} (with $s=8$, $t=2$) only yields $N_2 \geq 322$.

\subsection{Constructions from Symmetric BIBDs}

We next give a construction which potentially achieves similar behaviour as the 
random construction, using symmetric balanced
incomplete block designs (SBIBDs). A \emph{$(v,k,\lambda)$-balanced incomplete 
block design} (BIBD) is a pair $(X,\mathcal{A})$, where $X$ is a set of 
$v$ \emph{points} and $\mathcal{A}$ is a collection of $k$-subsets of $X$ 
called \emph{blocks}, such that every pair of points occurs in exactly $\lambda$ blocks.
Denote $b = |\mathcal{A}|$; it is well-known that $b = \lambda v (v-1) / (k(k-1))$.
It is also the case that every point occurs in exactly $r = bk / v = \lambda(v-1) / (k-1)$ blocks.
A BIBD is \emph{symmetric} if $v = b$. Equivalently, this condition can be expressed
as $r = k$ or $\lambda (v-1) = k(k-1)$.

Suppose $(X,\mathcal{A})$  is a  $(v,k,\lambda)$-BIBD. Denote
$X = \{x_i : 1 \leq i \leq v\}$ and $\mathcal{A} = \{A_j : 1 \leq j \leq b \}$.
The \emph{incidence matrix} of $(X,\mathcal{A})$ is the $v$ by $b$ $0-1$ matrix
$M = (m_{ij})$ where $m_{ij} = 1$ if $x_i \in A_j$, and $m_{ij} = 0$ if $x_i \not\in A_j$.

\begin{Lemma}
\label{SBIBDdet.thm}
Suppose $M$ is the incidence matrix of a symmetric $(v,k,\lambda)$-BIBD.
Then $M$ is invertible over $\eff_2$ if and only if $k$ is odd and $\lambda$ is even.
\end{Lemma}

\begin{proof}
It is well-known (see, e.g., \cite{CD}) that
$\det (M)$ is an integer and \[(\det (M))^2 = k^2 (k - \lambda)^{v-1}.\]
Reducing modulo 2, we see that $\det (M) \equiv 1 \bmod 2$ if and only if $k$ is odd and 
$\lambda$ is even.
\end{proof}

\begin{Theorem}
\label{SBIBD.thm}
Suppose $M$ is the incidence matrix of a symmetric $(v,k,\lambda)$-BIBD
where $k$ is odd and $\lambda$ is even.
Then 
\begin{equation}
\label{SBIBD.eq}
R_2(M) = \frac{k^2 - \lambda^2}{\binom{v}{2}}.
\end{equation}
\end{Theorem}

\begin{proof}
First, since $k$ is odd and $\lambda$ is even, $M$ is invertible over $\eff_2$ by Lemma \ref{SBIBDdet.thm}.
Consider two rows of $M$ and define $a_0, a_1, a_2, a_3$ as in the proof of Theorem 
\ref{upperbound}. Using the fact that $M$ is the incidence matrix of a 
symmetric $(v,k,\lambda)$-BIBD, it is not hard to see that 
$a_0 = v-2k+\lambda$, $a_1 = a_2  = k - \lambda$ and $a_3 = \lambda$.
Then we can compute 
\[ a_1a_2 + a_1a_3 + a_2a_3= 2\lambda(k-\lambda) + (k - \lambda)^2 = k^2 - \lambda^2.\]
From this, we have $N_2(M) = \binom{v}{2}(k^2 - \lambda^2)$ and (\ref{SBIBD.eq}) is easily derived.
\end{proof}

Let's try to figure out the best result that we could possibly obtain from Theorem \ref{SBIBD.thm}.
Suppose $k \approx cv$. Then from the equation $\lambda (v-1) = k(k-1)$, we see that
$\lambda \approx c^2 v$. Substituting into (\ref{SBIBD.eq}), we get
$R_2(M) \approx 2(c^2 - c^4)$. Now we of course have $0 \leq c \leq 1$, and 
the function $2(c^2 - c^4)$ is maximized when $c = \sqrt{1/2}$.
In this case, we would get $R_2(M) \approx 1/2$, more-or-less matching the random construction
from Section \ref{random.sec}. We have also guaranteed that the matrix $M$ is invertible.
Of course, we would require a suitable
SBIBD in order to get close to this bound.

We consider some examples to illustrate the application of Theorem \ref{SBIBD.thm}.

\begin{Example}
It is known \cite{CD} that there is a $(31,21,14)$-SBIBD. Noting that $21$ is odd and $14$ is even,
the incidence matrix of this design is invertible over $\eff_2$ 
by Lemma \ref{SBIBDdet.thm}. Observe that
$21/31$ is quite close to $\sqrt{1/2}$, so we expect a good result. 
Applying Theorem \ref{SBIBD.thm},
we get 
\[ R_2(M) = \frac{21^2 - 14^2}{\binom{31}{2}} = \frac{49}{93} \approx .5269.\]
\end{Example}

\begin{Example}
\label{40.ex}
There also exists $(40,27,18)$-SBIBD (see \cite{CD}). Noting that $27$ is odd and $18$ is even,
the incidence matrix of this design is invertible over $\eff_2$ 
by Lemma \ref{SBIBDdet.thm}. 
Applying Theorem \ref{SBIBD.thm},
we get 
\[ R_2(M) = \frac{27^2 - 18^2}{\binom{40}{2}} = \frac{27}{52} \approx .5192.\]
\end{Example}

\begin{Example}
A $(4m-1,2m-1,m-1)$-SBIBD is called a \emph{Hadamard design}. 
If $m$ is odd, then $\lambda = m-1$ is even. Certainly
$k = 2m-1$ is odd, so the incidence matrix $M$ is invertible, by Lemma \ref{SBIBDdet.thm}. 
These SBIBDs are known to exist for infinitely
many (odd) values of $m$, e.g., whenever $4m-1 \equiv 3 \bmod 8$ is
a prime or a prime power (see \cite{CD}). From the incidence matrix of such a BIBD, we obtain
\[ R_2(M) = \frac{(2m-1)^2 - (m-1)^2}{\binom{4m-1}{2}} \approx 3/8.\]
\end{Example}


\begin{Example}
Here we make use of  a classic result based on difference sets.
Suppose $q = 4t^2+9$ is prime and $t$ is odd. In this situation, it was shown
by E.\ Lehmer  that the 
quartic residues modulo $q$, together with $0$, form a difference set
which generates a $(q,(q+3)/4,(q+3)/16)$-SBIBD (e.g., see \cite[p.\ 116]{CD}).
If we complement  this design (i.e., we replace all 0's by 1's 
and all 1's by 0's in the incidence matrix), the result is a $(q,3(q-1)/4,3(3q-7)/16)$-SBIBD.
This SBIBD will have $k$ odd and $\lambda$ even, so its incidence matrix 
$M$ is invertible, by Lemma \ref{SBIBDdet.thm}. 
The first example is obtained when $t=5$, yielding 
\[ R_2(109) \geq \frac{329}{654}.\]
Asymptotically, from (\ref{SBIBD.eq}), we obtain 
\[ R_2(M) \approx \frac{63}{128}\] if there exist sufficiently large $q$
of the desired form. However, it is a famous unsolved conjecture that
there exist infinitely many primes of the form $x^2 + 9$, so we are not in
a position to claim that this asymptotic result holds.
\end{Example}

The following theorem generalizes Example \ref{40.ex}.

\begin{Theorem}
\label{geom.thm}
Suppose $m$ is a positive integer and $s = (3^{m+1} - 1)/2$.
Then \[ R_2(s) \geq  \frac{40 \times 3^{2m-3}}{(3^{m+1} - 1)(3^{m} - 1)}.\]
\end{Theorem}

\begin{proof}
The points and hyperplanes of the $m$-dimensional projective geometry over
$\eff_3$ yield a 
$\left(\frac{3^{m+1} - 1}{2}, \frac{3^{m} - 1}{2}, \frac{3^{m-1} - 1}{2}\right)$-SBIBD.
If we complement this design, we get a
$\left(\frac{3^{m+1} - 1}{2}, 3^m, 2\times 3^{m-1}\right)$-SBIBD.
This design has $k$ odd and $\lambda$ even, so we can apply Theorem \ref{SBIBD.thm}.
The result is that
\[ R_2\left(\frac{3^{m+1} - 1}{2}\right) \geq 
\frac{(3^{m})^2 - (2\times 3^{m-1})^2}{\binom{\frac{3^{m+1} - 1}{2}}{2}}
= \frac{40 \times 3^{2m-3}}{(3^{m+1} - 1)(3^{m} - 1)}.\]
\end{proof}

Let's examine the asymptotic behaviour of the result proven in Theorem \ref{geom.thm}.
The SBIBD has $k \approx 2v/3$ and $\lambda \approx 4v/9$.
It then follows from (\ref{SBIBD.eq}) that
\[ R_2(M)  = \frac{k^2 - \lambda^2}{\binom{v}{2}} \approx 
2\left(\left(\frac{2}{3}\right)^2 - \left(\frac{4}{9}\right)^2\right) = \frac{40}{81}.\]

Therefore, we obtain the following corollary.
\begin{Corollary}
It holds that $\limsup _{s \rightarrow \infty} R_2(s) \geq \frac{40}{81}$.
\end{Corollary}
We note that $40/81 \approx .494$. So there is a  gap between our
upper and lower asymptotic bounds on $2$-density, which are respectively
.625 (from Corollary \ref{Cor-15}) and .494 (and of course the lower bound only has been proven to hold 
within a certain infinite class
of examples).

\subsection{Constructions using Cyclotomy}

We now look at constructions using cyclotomy.
Let $p = 4f+1$ be prime, where $f$ is even, and let $\nu \in {\eff_p}^*$ be a primitive element.
Let $C_0 = \{\nu^{4i}: 0 \leq i \leq f-1 \}$; this is the unique subgroup of ${\eff_p}^*$ having order $f$. The multiplicative cosets of $C_0$ are $C_j = \nu^j C_0$, for $j = 0,1,2,3$. 
These cosets are often called \emph{cyclotomic classes}. 

We now construct a $p$ by $p$ $0-1$ matrix $M' = (m_{ij})$ from $C_0$.
The rows and columns of $M'$ are indexed by $\eff_p$,
and $m_{ij} = 1$ if and only if $j - i \in C_0$.
Note that the $i$th row of $M'$ is the incidence vector of $i + C_0$.
Finally, define $M$ to be the complement of $M'$ (i.e., replace all 1's by 0's and vice versa).

We will now compute the number of invertible $2$ by $2$ submatrices of $M$.
Consider rows $i_1$ and $i_2$ of $M$. It is obvious that the number of 
invertible $2$ by $2$ submatrices contained in these two rows is the same as the
number of invertible $2$ by $2$ submatrices contained in rows $0$ and $d$,
where $d = i_1 - i_2$. We can compute this number if we can determine the 
number $n_{d}$ of columns
$c$ such that $m_{0c} = m_{d c} = 1$. 
It is clear that 
\[n_{d} = | C_0 \cap (d + C_0)| .\]
However, \[| C_0 \cap (d + C_0)| = | d^{-1} C_0 \cap (1 + d^{-1} C_0)|.\]
Now, $d^{-1} C_0 = C_j$, for some $j$, $0 \leq j \leq 3$,
so \[n_{d} = | C_j \cap (1 + C_j)| \]
for this particular value of $j$.
This quantity is a \emph{cyclotomic number of order $4$} and is denoted by $(j,j)$.

We will make use of the following theorem from \cite{KaRa}.

\begin{Theorem}\cite[Theorem 1]{KaRa}
\label{cyc.thm}
Suppose $p = 4f+1$ is prime and $f$ is even. Let $\nu \in \eff_q$ be a primitive element.
Let $p = \alpha^2 + \beta^2$, where $\alpha \equiv 1 \bmod{4}$ and
$\nu^{f} \equiv \alpha/\beta \bmod{p}$. Then the cyclotomic numbers 
$(j,j)$ ($0 \leq j \leq 3$),
are as follows:
\begin{align*}
(0,0)  &= A_0 = \frac{p-11-6\alpha}{16} = \frac{4f-10-6\alpha}{16}\\
(1,1)  &= A_1 = \frac{p-3+2\alpha-4\beta}{16} = \frac{4f-2+2\alpha-4\beta}{16}\\
(2,2)  &= A_2 = \frac{p-3+2\alpha}{16} = \frac{4f-2+2\alpha}{16}\\
(3,3)  &= A_3 = \frac{p-3+2\alpha+4\beta}{16} = \frac{4f-2+2\alpha+4\beta}{16}.
\end{align*}
\end{Theorem}

\begin{Remark}
A prime $p \equiv 1 \bmod 4$ can be expressed as the sum of two squares in a unique manner.
If $p = \alpha^2 + \beta^2$, then one of $\alpha, \beta$ is odd and the other is even. 
So without loss of generality we can take $\alpha$ to be odd and $\beta$ to be even. 
In this way, $\alpha$ and $\beta$ are determined up to their signs. The condition
$\alpha \equiv 1 \bmod{4}$ now determines $\alpha$ uniquely, and, similarly, 
$\nu^{f} \equiv \alpha/\beta \bmod{p}$ determines $\beta$ uniquely.
\end{Remark}

Now we can compute the number of $2$ by $2$ submatrices contained in rows
$i_1$ and $i_2$ of $M$. Again we define $a_0, a_1, a_2, a_3$ as in the proof of Theorem 
\ref{upperbound}. Recalling that $M$ is the complement of $M'$, we have 
\[a_1 = a_2 = f - (j,j)\]
and \[a_3 = p - 2f + (j,j) = 2f+1 + (j,j),\] where $(i_1 - i_2)^{-1} C_0 = C_j$.
Thus we obtain
\[ a_1a_2 + a_1a_3 + a_2a_3 = 5f^2 + 2f - (j,j) (4f + 2 + (j,j)).\] 

As we consider all $\binom{p}{2}$  pairs $\{i_1,i_2\}$ with $i_1 \neq i_2$,
we see that $(j,j)$ takes on each of the four possible values $A_i$ ($1 \leq i \leq 4$)
one quarter of the time.
Therefore the total number of invertible $2$ by $2$ submatrices in $M$ is
\begin{eqnarray*}
 \lefteqn{
 \binom{p}{2} \sum_{i=0}^{3} \left( \frac{5f^2 + 2f - A_i (4f + 2 + A_i)}{4} \right)
 }\\
 & = & \binom{p}{2} \frac{252f^2 +168f +25 -3\alpha^2-2\beta^2-6\alpha}{64},
\end{eqnarray*}
where the last line is obtained by applying the formulas given in Theorem \ref{cyc.thm}.

\begin{Example}
\label{ex17}
Suppose $p = 17 = 4 \times 4 + 1$. 
Then $\nu = 3$ is a primitive element.
Since $17 = 1^2 + 4^2$, we have $\alpha = 1$ and $\beta \in \{4,13\}$.
We compute $3^{4} \equiv 13 \bmod 17$. Since $1/4 \equiv 13 \bmod{17}$,
we have $\beta = 4$.
The cyclotomic classes are 
\begin{eqnarray*}
C_0 &=& \{1, 13,16 ,4\}\\
C_1 &=& \{3,5, 14,12\}\\
C_2 &=& \{9,15,8,2\}\\
C_3 &=& \{10,11,7,6\}.
\end{eqnarray*}
The cyclotomic numbers can now be computed from Theorem \ref{cyc.thm}; they are
\begin{align*}
(0,0)  &= A_0 = \frac{17-11-6}{16} = 0\\
(1,1)  &= A_1 = \frac{17-3+2-4\times 4}{16} = 0\\
(2,2)  &= A_2 = \frac{17-3+2}{16} = 1\\
(3,3)  &= A_3 = \frac{17-3+2+4\times 4}{16} = 2.
\end{align*}
The total number of invertible $2$ by $2$ submatrices in $M$ is
$9962$.
\end{Example}

\begin{table}[tb]
\caption{Examples from Cyclotomy}
\label{cyc.tab}
\[
\begin{array}{rrrrrc}
p & f & \alpha & \beta & N_2(M) & R_2(M) \\ \hline
17 &  3 &  1 &  4 &  9962 &  .53860\\
97 &  5 &  9 &  -4 &  10831020 &  .49962\\
193 &  5 &  -7 &  12 &  170314008 &  .49613\\
241 &  7 &  -15 &  4 &  414228390 &  .49527\\
401 &  3 &  1 &  -20 &  3177945050 &  .49408\\
433 &  5 &  17 &  -12 &  4320175230 &  .49388\\
449 &  3 &  -7 &  20 &  4995836216 &  .49388
\end{array}
\]
\end{table}

It remains to consider the invertibility of the matrices $M$ constructed above. 
The matrices in question are cyclic. Suppose a $p$ by $p$ cyclic $0-1$ matrix $M$ has as its initial row
the vector $(m_0, \dots , m_{p-1})$. We associate with this vector the polynomial
\[m(x) = \sum_{i=0}^{p-1} m_i x^i \in \zed_2[x].\] It is easy to see that $M$ is invertible
if and only if $\gcd( m(x), x^p-1) = 1$. In this case, the inverse $m^{-1}(x)$ of $m(x)$ is defined 
in the quotient ring $\zed_2[x] / (x^p-1)$. The cyclic matrix whose first row is determined by
$m^{-1}(x)$  will in fact be the inverse of $M$.
Therefore, to determine the invertibility of $M$, we just need to do a 
gcd computation. 

\begin{Example}
Let $p=17$. From Example \ref{ex17}, we have $C_0 = \{1, 13,16 ,4\}$.
The first row of $M$ is 
\[ 1 \: 0 \: 1 \: 1 \: 0 \: 1 \: 1 \: 1 \: 1 \: 1 \: 1 \: 1 \: 1 \: 0 \: 1 \: 1 \: 0 ,\]
which corresponds to the polynomial
\[ m(x) = 1 + x^2 + x^3 +  x^5 + x^6 + x^7+ x^8 + x^9 + x^{10}+ x^{11}+x^{12}+x^{14}+x^{15}.\]
The inverse of $m(x)$ is
\[ m^{-1}(x) = 1 + x +   x^3 + x^4 + x^5 + x^6 + x^7+ x^{10}+ x^{11}+x^{12}+ x^{13} +x^{14}+x^{16}.\]
\end{Example}

By Dirichlet's Theorem, there are an infinite number of primes $p \equiv 1 \bmod 8$.
However, we do not have a theoretical criterion to 
determine if  a given matrix $M$ in this class of examples is invertible.
Therefore, we cannot prove that there are an infinite number of examples of this type.
However, by computing gcds, as described above, we determined all the invertible matrices $M$
of order less than $500$ that can be constructed by this method.
Some data about these matrices is presented in 
Table \ref{cyc.tab}. Another observation is that, if this is in fact an infinite class, 
then it can be shown that the density of these examples approaches $63/128 \approx .492$ as $f$ approaches infinity.

\subsection{Values and Bounds on $N_2(s)$ for Small $s$}
We summarize our upper and lower bounds on $N_2(s)$ for  $s \leq 12$ in Table \ref{tab2}.
For the cases $s = 5,6,7,8$, we have exact values of $N_2(s)$ that are obtained from exhaustive computer searches
For $s=9$, our lower bound is obtained from a partial (uncompleted) exhaustive search.
For $s = 10,11,12$, the lower bounds come from randomly constructed matrices.
All of these matrices are presented in Appendix A.

\begin{table}[tb]
 \caption{Values and Bounds on $N_2(s)$ for small $s$}
 \label{tab2}
 \[
 \begin{array}{c|c|c}
 s & N_2(s)  & \text{justification} \\ \hline
2 & N_2(2) = 1   & \text{Lemma \ref{2x2.lem}} \\
3 & N_2(3) = 7 & \text{Example \ref{exam3}} \\
4 &  N_2(4) = 30 & \text{Example \ref{exam4}} \\
5 & N_2(5) = 70 & \text{exhaustive search (Example \ref{example5})}\\
6 & N_2(6) = 150  & \text{exhaustive search (Example \ref{example6})}\\
7 & N_2(7) = 287  & \text{exhaustive search  (Example \ref{example7})}\\
8 & N_2(8) = 485 & \text{exhaustive search  (Example \ref{example8})}\\
9 & N_2(9) \geq 783  & \text{Example \ref{example9}}\\
10 & N_2(10) \geq 1194   & \text{Example \ref{example10}}\\
11 & N_2(11) \geq 1744   & \text{Example \ref{example11}}\\
12 & N_2(12) \geq 2448  & \text{Example \ref{example12}}
 \end{array}
 \]
\end{table}

\section{General Transforms}
\label{general.sec}

In this section, we examine
general (i.e., linear or nonlinear) AONT over an arbitrary alphabet,
extending some results from \cite{St} in a straightforward manner.

Let $A$ be an $N$ by $k$ array whose entries are elements chosen from 
an alphabet $X$ of order $v$.  
We will refer to $A$ as an {\em $(N,k,v)$-array}. 
Suppose the columns of $A$ are
labelled by the elements in the set $C = \{ 1, \dots , k\}$.  
Let $D \subseteq C$, and define $A_D$ to be the array obtained from $A$
by deleting all the columns $c \notin D$.
We say that $A$ is
{\em unbiased} with respect to $D$ if the rows of
$A_D$ contain every $|D|$-tuple of elements of $X$ 
exactly $N / v^{|D|}$ times.

The following result characterizes $(t,s,v)$-AONT in terms of
arrays that are unbiased with respect to
certain subsets of columns.

\begin{Theorem}
\label{equiv}
A $(t,s,v)$-AONT is equivalent to a $(v^s,2s,v)$-array
that is unbiased with respect to the following subsets of columns:
\begin{enumerate}
\item $\{1, \dots , s\}$,
\item $\{s+1, \dots , 2s\}$, and
\item $I \cup \{ s+1, \dots , 2s\} \setminus J$, 
         for all $I \subseteq \{1,\dots , s\}$ with $|I| = t$ and all 
         $J \subseteq \{s+1,\dots , 2s\}$ with $|J| = t$.
\end{enumerate}
\end{Theorem}

\begin{proof}
Let $A$ be the hypothesized $(v^s,2s,v)$-array on alphabet $X$, $|X| = v$.
We construct $\phi : X^s \rightarrow X^s$ as follows:
for each row $(x_1, \dots , x_{2s})$ of $A$, define
\[ \phi (x_1, \dots , x_{s}) = (x_{s+1}, \dots , x_{2s}).\]
The function $\phi$ is easily seen to be a $(t,s,v)$-AONT.

Conversely, suppose $\phi$ is a $(t,s,v)$-AONT.
Define an array $A$ whose rows consist of all $v^s$ $2s$-tuples $(x_1, \dots , x_{2s})$, where
$\phi (x_1, \dots , x_{s}) = (x_{s+1}, \dots , x_{2s})$.
Then $A$ is the desired $(v^s,2s,v)$-array.
\end{proof}

An OA$(s,k,v)$ (an {\em orthogonal array})
is a $(v^s,k,v)$-array that is unbiased with respect to
any subset of $s$ columns.
The following corollary of Theorem \ref{equiv} is immediate.

\begin{Corollary}
If there exists an OA$(s,2s,v)$, then there exists a $(t,s,v)$-AONT
for all $t$ such that $1 \leq t \leq s$.
\end{Corollary}

\section{Summary}
\label{summary.sec}

We have  initiated a study of all-or-nothing transforms where we consider
the security of $t$ inputs when $s-t$ outputs are fixed or known.
We focussed on the case $t=2$, for linear transformations defined over $\eff_2$.
Here one fundamental problem is to determine the maximum $2$-density, which we denoted by $R_2(s)$.
We ask if $\lim _{s \rightarrow \infty} R_2(s)$ exists, and if so, what its value is.
Our results establish that, if this limit $L$ exists, then
$.494  \leq L \leq .625$. 

It should be possible to adapt the techniques used in this
paper to deal with cases where $t > q$, or to analyze tranforms over
$\eff_q$ for fixed prime powers $q > 2$.

\section*{Acknowledgements}

D.\ R.\ Stinson would like to thank Jeroen van de Graaf for rekindling his
interest in AONT. Thanks also to Hugh Williams for providing some 
relevant number-theoretic information, and to Yuying Li and Henry Wolkowicz
for useful discussions about quadratic programming.

\appendix

\section{Matrices Yielding Good or Exact Lower Bounds for $N_2(M)$}

\begin{Example}
\label{example5}
	An invertible $5$ by $5$ matrix having $70$  invertible $2\times 2$ submatrices:
	\[
	M_{5\times 5}=\left(\begin{array}{c c c c c}
	0  &  0  &  1  &  1  &  1\\
	0  &  1  &  0  &  1  &  1\\
	1  &  0  &  0  &  1  &  1\\
	1  &  1  &  1  &  0  &  1\\
	1  &  1  &  1  &  1  &  0	
	\end{array}\right).
	\]
\end{Example}

\begin{Example}
\label{example6}
	An invertible $6$ by $6$ matrix having $150$ invertible $2\times 2$ submatrices:
	\[
	M_{6\times 6}=\left(\begin{array}{c c c c c c}
	0  &  0  &  1  &  1  &  1  &  1\\
	0  &  1  &  0  &  1  &  1  &  1\\
	1  &  0  &  1  &  0  &  1  &  1\\
	1  &  1  &  0  &  1  &  0  &  1\\
	1  &  1  &  1  &  0  &  0  &  1\\
	1  &  1  &  1  &  1  &  1  &  0
	\end{array}\right).
	\]
\end{Example}

\begin{Example}
\label{example7}

	An invertible $7$ by $7$ matrix having $287$ invertible $2\times 2$ submatrices:
	\[
	M_{7\times 7}=\left(\begin{array}{c c c c c c c}
	0  &  0  &  1  &  1  &  1  &  1  &  1\\
	0  &  1  &  0  &  1  &  1  &  1  &  1\\
	1  &  0  &  1  &  0  &  1  &  1  &  1\\
	1  &  1  &  0  &  1  &  0  &  1  &  1\\
	1  &  1  &  1  &  0  &  1  &  0  &  1\\
	1  &  1  &  1  &  1  &  0  &  1  &  0\\
	1  &  1  &  1  &  1  &  1  &  0  &  0
	\end{array}\right)
	\]
\end{Example}

\begin{Example}
\label{example8}
	An invertible $8$ by $8$ matrix having $485$ invertible $2\times 2$ submatrices:
	\[
	M_{8\times 8}=\left(\begin{array}{c c c c c c c c}
	0  &  0  &  0  &  1  &  1  &  1  &  1  &  1\\
	0  &  1  &  1  &  0  &  1  &  1  &  1  &  1\\
	0  &  1  &  1  &  1  &  0  &  1  &  1  &  1\\
	1  &  0  &  1  &  0  &  1  &  1  &  1  &  1\\
	1  &  1  &  0  &  1  &  0  &  1  &  1  &  1\\
	1  &  1  &  1  &  1  &  1  &  0  &  0  &  1\\
	1  &  1  &  1  &  1  &  1  &  0  &  1  &  0\\
	1  &  1  &  1  &  1  &  1  &  1  &  0  &  0
	\end{array}\right)
	\]
\end{Example}

\begin{Example}
\label{example9}
An invertible $9$ by $9$ matrix having $783$ invertible $2\times 2$ submatrices:
	\[
	M_{9\times 9}=\left(\begin{array}{c c c c c c c c c}
 0  &  0  &  0  &  1  &  1  &  1  &  1  &  1  &  1\\                                          
 0  &  1  &  1  &  0  &  0  &  1  &  1  &  1  &  1\\
 0  &  1  &  1  &  1  &  1  &  0  &  0  &  1  &  1\\
 1  &  0  &  1  &  0  &  1  &  0  &  1  &  1  &  1\\
 1  &  0  &  1  &  1  &  0  &  1  &  1  &  0  &  1\\
 1  &  1  &  0  &  1  &  0  &  1  &  0  &  1  &  1\\
 1  &  1  &  0  &  1  &  1  &  1  &  1  &  0  &  0\\
 1  &  1  &  1  &  0  &  1  &  1  &  0  &  0  &  1\\
 1  &  1  &  1  &  1  &  1  &  0  &  1  &  1  &  0	
	\end{array}\right)
	\]
\end{Example}

\begin{Example}
\label{example10}
	An invertible $10$ by $10$ matrix having $1194$ invertible $2\times 2$ submatrices:
	\[
	M_{10\times 10}=\left(\begin{array}{c c c c c c c c c c}
	0  &  0  &  0  &  1  &  1  &  1  &  1  &  1  &  1  &  1\\
	0  &  1  &  1  &  0  &  1  &  1  &  0  &  1  &  1  &  1\\
	0  &  1  &  1  &  1  &  0  &  1  &  1  &  0  &  1  &  1\\
	1  &  0  &  1  &  1  &  1  &  0  &  1  &  1  &  0  &  1\\
	1  &  1  &  0  &  1  &  1  &  1  &  0  &  1  &  1  &  1\\
	1  &  1  &  1  &  0  &  1  &  1  &  1  &  0  &  1  &  1\\
	1  &  0  &  1  &  1  &  0  &  1  &  1  &  1  &  0  &  1\\
	1  &  1  &  0  &  1  &  1  &  0  &  1  &  1  &  1  &  0\\
	1  &  1  &  1  &  0  &  1  &  1  &  0  &  1  &  1  &  0\\
	1  &  1  &  1  &  1  &  1  &  1  &  1  &  0  &  0  &  0
	\end{array}\right)
	\]
\end{Example}

\begin{Example}
\label{example11}
An invertible $11$ by $11$ matrix having $1744$ invertible $2\times 2$ submatrices:
	\[
	M_{11\times 11}=\left(\begin{array}{c c c c c c c c c c c}
	0  &  0  &  0  &  1  &  1  &  1  &  1  &  1  &  1  &  1  &  0\\
	0  &  1  &  1  &  0  &  1  &  0  &  1  &  1  &  1  &  0  &  1\\
	0  &  1  &  1  &  1  &  0  &  1  &  1  &  0  &  1  &  1  &  1\\
	1  &  0  &  1  &  1  &  1  &  0  &  1  &  1  &  0  &  1  &  1\\
	1  &  1  &  0  &  1  &  1  &  1  &  0  &  1  &  1  &  1  &  1\\
	1  &  1  &  1  &  0  &  1  &  1  &  1  &  0  &  1  &  0  &  1\\
	1  &  1  &  0  &  1  &  0  &  1  &  1  &  1  &  0  &  1  &  1\\
	1  &  0  &  1  &  1  &  1  &  0  &  1  &  1  &  1  &  0  &  1\\
	1  &  1  &  1  &  1  &  0  &  1  &  0  &  1  &  1  &  1  &  0\\
	1  &  0  &  1  &  0  &  1  &  1  &  1  &  0  &  1  &  1  &  0\\
	0  &  1  &  1  &  1  &  1  &  1  &  1  &  1  &  0  &  0  &  0	
	\end{array}\right)
	\]
\end{Example}

\begin{Example}
\label{example12}
An invertible $12$ by $12$ matrix having $2448$ invertible $2\times 2$ submatrices:
	\[
	M_{12\times 12}=\left(\begin{array}{c c c c c c c c c c c c}
	1  &  0  &  1  &  1  &  1  &  1  &  1  &  1  &  1  &  0  &  1  &  0\\
	1  &  1  &  1  &  1  &  1  &  0  &  1  &  1  &  1  &  1  &  1  &  0\\
	1  &  1  &  1  &  0  &  1  &  0  &  1  &  1  &  0  &  1  &  1  &  1\\
	0  &  0  &  1  &  0  &  1  &  1  &  1  &  0  &  1  &  1  &  1  &  1\\
	0  &  0  &  1  &  1  &  1  &  1  &  1  &  1  &  0  &  1  &  0  &  1\\
	1  &  1  &  0  &  1  &  1  &  0  &  0  &  1  &  1  &  1  &  1  &  1\\
	0  &  1  &  1  &  0  &  0  &  1  &  1  &  0  &  1  &  1  &  1  &  1\\
	0  &  1  &  0  &  1  &  1  &  1  &  0  &  1  &  1  &  1  &  0  &  1\\
	0  &  1  &  1  &  1  &  0  &  1  &  1  &  1  &  1  &  1  &  1  &  0\\
	1  &  1  &  1  &  1  &  1  &  1  &  0  &  1  &  0  &  0  &  1  &  1\\
	1  &  1  &  1  &  0  &  1  &  1  &  1  &  1  &  1  &  1  &  0  &  1\\
	1  &  1  &  0  &  1  &  1  &  1  &  1  &  0  &  1  &  0  &  0  &  1	
	\end{array}\right)
	\]
\end{Example}

\end{document}